\documentclass{article}
\usepackage[utf8]{inputenc}
\usepackage{amsmath}
\usepackage{amsthm}
\usepackage{amssymb}
\usepackage{caption}
\usepackage{subcaption}
\usepackage{graphicx}

\newtheorem{theorem}{Theorem}[section]
\newtheorem{corollary}[theorem]{Corollary}
\newtheorem{example}[theorem]{Example}
\newtheorem{lemma}[theorem]{Lemma}
\newtheorem{prop}[theorem]{Proposition}

\theoremstyle{definition}
\newtheorem{definition}[theorem]{Definition}
\newtheorem{remark}[theorem]{Remark}
\usepackage[capitalize]{cleveref}
	\crefname{claim}{Claim}{Claims}
	\Crefname{claim}{Claim}{Claims}
	\crefname{app-corollary}{Corollary}{Corollaries}
	\Crefname{app-corollary}{Corollary}{Corollaries}
	\crefname{app-definition}{Definition}{Definitions}
	\Crefname{app-definition}{Definition}{Definitions}
	\crefname{figure}{Figure}{Figures}
	\Crefname{figure}{Figure}{Figures}
	\crefname{lemma}{Lemma}{Lemmata}
	\Crefname{lemma}{Lemma}{Lemmata}
	\crefname{app-lemma}{Lemma}{Lemmata}
	\Crefname{app-lemma}{Lemma}{Lemmata}
	\crefname{app-proposition}{Proposition}{Proposition}
	\Crefname{app-proposition}{Proposition}{Proposition}
	\crefname{app-theorem}{Theorem}{Theorems}
	\Crefname{app-theorem}{Theorem}{Theorems}

\newcommand{\Comp}{\operatorname{Comp}}
\newcommand{\Des}{\operatorname{Des}}
\newcommand{\sd}{\mathfrak{d}}

\title{A Generalization of Descent Polynomials}
\author{{\Large Angel Raychev}
\and
 Electronic address: \texttt{angel11drakon@gmail.com}
\and
125 High school, Sofia, Bulgaria}
\date{September, 2021}

\begin{document}

\maketitle

\begin{abstract}
The notion of a \textit{descent polynomial}, a function in enumerative combinatorics that counts permutations with specific properties, enjoys a revived recent research interest due to its connection with other important notions in combinatorics, viz. \textit{peak polynomials} and \textit{symmetric functions}. We define the function $\sd^{m}(I,n)$ as a generalization of the descent polynomial and we prove that for any positive integer $m$, this function is a polynomial in $n$ for sufficiently large $n$ (similarly to the descent polynomial). We obtain an explicit formula for $\sd^{m}(I,n)$ when $m$ is sufficiently large. We look at the coefficients of $\sd^{m}(I,n)$ in different falling factorial bases. We prove the positivity of the coefficients and discover a combinatorial interpretation for them. This result is similar to the positivity result of Diaz-Lopez et al. for the descent polynomial.
\end{abstract}

\section{Introduction}

The \textit{descent polynomial} is a function in enumerative combinatorics that counts permutations with specific properties. It originates from 1915, when MacMahon \cite{macmahon2001combinatory} introduced it in his book \textit{Combinatory Analysis}. However, in the following years, there was not much published about the descent polynomial. It was only in 2017 when the topic was revisited by Diaz-Lopez, Harris, Insko, Omar, and Sagan \cite{diaz2019descent}. In their paper, Diaz-Lopez et al. gave two recurrence relations \cite[Section 2]{diaz2019descent} for the descent polynomial, looked at its coefficients, and investigated its roots. In 2018, Bencs \cite{bencs2021some} answered some questions about the coefficients of the descent polynomial, which were stated earlier by Diaz-Lopez et al. Since then, there were several other publications looking at some generalizations of the descent polynomial. However, one intuitive generalization about the multiplicity of the elements has not been studied so far. In this paper, we look at this generalization. 

First, let us define the original descent polynomial that was studied by \ \ \ Diaz-Lopez et al. in 2017. Let $I$ be a finite set of positive integers. We consider permutations of the set $\{1,2,\dots,n\}$. The \textit{descent polynomial} is equal to the number of permutations for which the set of all positions in the permutation where the corresponding elements are bigger than the next ones is exactly $I$. We study the following generalization, which has not been studied yet: instead of the set $\{1,2,\dots,n\}$, we are considering permutations of the multiset $\{\underbrace{1,\dots,1}_{m},\underbrace{2,\dots,2}_{m},\dots, \underbrace{n,\dots,n}_{m}\}$, where each number is of multiplicity $m$. In this case we use the notation $\sd^{m}(I,n)$. Note that the descent polynomial is equal to $\sd^{1}(I,n)$. Similarly to the descent polynomial, we prove that $\sd^{m}(I,n)$ is also a polynomial in $n$ (Theorem \ref{theo5}).

One interesting observation we make is that as $m$ increases, $\sd^{m}(I,n)$ also increases, but from some point onward $\sd^{m}(I,n)$ stabilizes. So, we introduce the notation $d^{\infty}(I,n)$ for the stabilized value of $\sd^{m}(I,n)$. We give an explicit formula for computing $d^{\infty}(I,n)$ (Theorem \ref{theo6}).
\[d^{\infty}(I,n)=\sum\limits_{A\in \Comp(t)}(-1)^{t-s}\binom{n-1+\sigma_{1}}{\sigma_{1}} \dots \binom{n-1+\sigma_{s-1}}{\sigma_{s-1}}
\left(\binom{n-1+\sigma_{s}}{\sigma_{s}}-1\right),\]
where the summation is over all compositions $A\in\Comp(t)$, $t$ is the size of $I$, and $\sigma_{1}, \sigma_{2},\ldots, \sigma_{s}$ depend on $A$ and $I$.

We look at the coefficients of $d^{\infty}(I,n)$ in falling factorial bases of the type $\left(\binom{n+k}{i}\right)_{i=0}^{\infty}$. If $k\leq -1$, all coefficients are non-negative integers. If $k\geq 0$, all coefficients are integers, but there exists at least one negative. For $k=-1$, we give a combinatorial interpretation to the coefficients of $d^{\infty}(I,n)$ and for $k=0$, we find the exact values of some of the coefficients.

The paper is organized in the following way. In Section \ref{pre} we introduce some useful notations.  In Section \ref{gen} we show some general properties of $\sd^{m}(I,n)$. Then in Section \ref{pol} we prove that $\sd^{m}(I,n)$ is a polynomial in $n$ for sufficiently large $n$. In Section \ref{jac} we give an alternative proof for the polynomiality of $\sd^{m}(I,n)$. In Section \ref{form} we introduce the function $d^{\infty}(I,n)$ and we give an explicit formula for computing it. In Section \ref{coef} we look at the coefficients of $d^{\infty}(I,n)$.

\section{Preliminaries}\label{pre}

In this section, we introduce some notations which are used throughout the paper.  

For the rest of the paper, we assume that $n$ and $m$ are positive integers and $I$ is a finite set of positive integers. If $I$ is non-empty, we let $I=\{\alpha_{1},\alpha_{2},\ldots,\alpha_{t}\}$, where $\alpha_{1}<\alpha_{2}<\cdots<\alpha_{t}$. By $I^{-}$ we denote the set $I$ without its biggest element $\alpha_{t}$, so $I^{-}=\{\alpha_{1},\alpha_{2},\ldots,\alpha_{t-1}\}$. Also, for a non-empty set $I$, we let $L$ be the length of the longest sequence of consecutive numbers in $I$. For example, if $I=\{2,3,5,7,10,11,12\}$, $L=3$, because $10,11,12$ is the longest sequence of consecutive numbers in $I$.  

Let $\Comp(t)$ be the set of all compositions of $t$ and let $\Comp(t,m)$ be the set of all compositions of $t$, which do not have elements bigger than $m$. Recall that a composition of a positive integer $t$ is an ordered sequence of positive integers with sum $t$. 

\begin{definition}
Let $v=(v_{1},v_{2},v_{3},\dots, v_{\ell}$) be a finite sequence of positive integers. We define
$$\Des(v) = \{i\in\{1,2,3,\dots,\ell-1\} : v_{i}>v_{i+1}\}$$
to be the \textit{descent set} of the sequence $v$.
\end{definition}

\begin{example}
If $v=(1,3,2,6,1,1,9,3)$, then $\Des(v)=\{2,4,7\}.$
\end{example}

Next, we formally define the generalization of the descent polynomial, which we are considering in this paper.

\begin{definition}
Let $S_{n}^{(m)}$ be the set of all permutations of the multiset 
$$\{\underbrace{1,1,\dots,1}_{m},\underbrace{2,2,\dots,2}_{m},\dots,\underbrace{n,n,\dots,n}_{m}\}.$$
We define the function $\sd^{m}(I,n)$ such that
$$\sd^{m}(I,n) = \# \{w\in S_{n}^{(m)} : \Des(w)=I\}.$$
\end{definition}

\begin{example}
If $n=3$, $m=2$, and $I=\{2\}$, then $\sd^{2}(\{2\},3)=5$, because there are $5$ such permutations:
\[(1,2,1,2,3,3) \ \ (1,3,1,2,2,3) \ \ (2,2,1,1,3,3) \ \ (2,3,1,1,2,3) \ \ (3,3,1,1,2,2).\]
\end{example}

\begin{remark}\label{rem2}
A different way of visualizing the permutations we are considering is by semistandard Young tableaux of ribbon shape. Recall that a semistandard Young tableau is a Young diagram filled with positive integers such that the numbers are weakly increasing from left to right and are strictly decreasing from bottom to top. Recall also that a ribbon shape is a connected skew shape without any $2\times 2$ squares in it. The shape of the ribbon is determined by the descent set $I$ and there are $n m$ cells in it.
\end{remark}

\begin{example}
In Figure \ref{fig7} we show an example of a semistandard Young tableau of ribbon shape for $I=\{4,8,9\}$, $n=5$, and $m=3$. 

\begin{figure}[h]
    \centering
    \captionsetup{justification=centering}
    \includegraphics[scale=0.7]{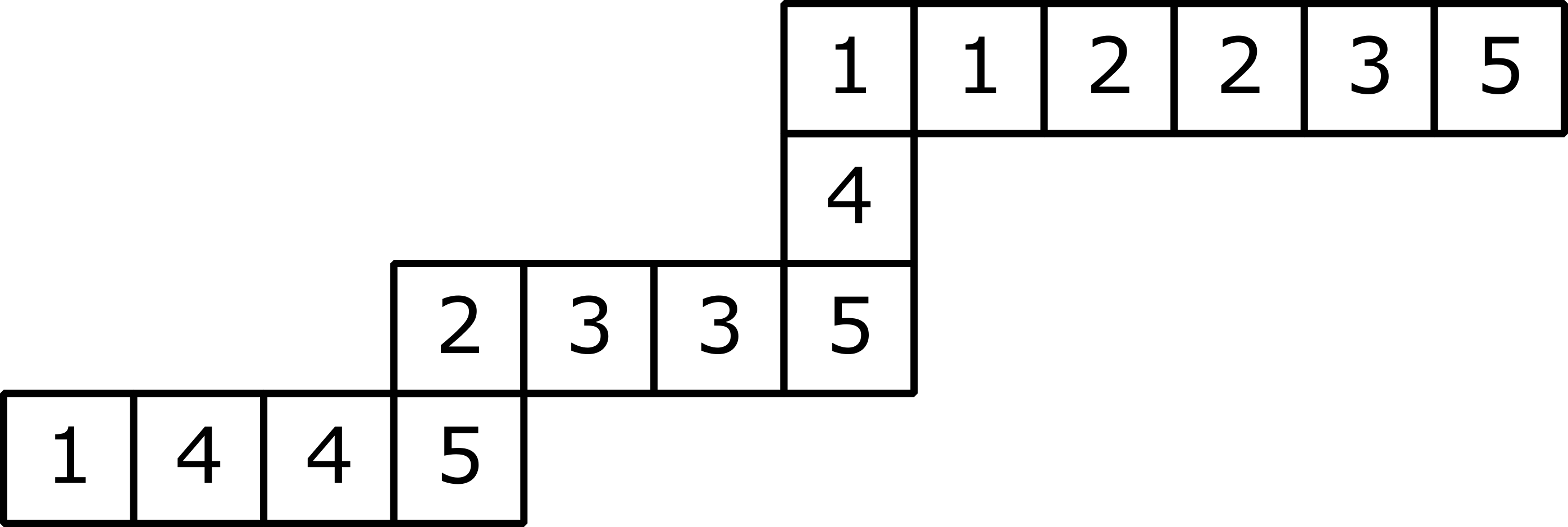}
    \caption{A semistandard Young tableau for\\
    $I=\{4,8,9\}$, $n=5$, and $m=3$.}
    \label{fig7}
\end{figure}
\end{example}

\begin{remark}\label{rem3}
Recall that each Young diagram is determined by a unique partition $\lambda=[\lambda_{1},\lambda_{2},\ldots,\lambda_{s}]$ in a such way that in the $i$-th row of the diagram there are $\lambda_{i}$ squares. For example, if $\lambda=[9,7,3,3,1]$, the corresponding Young diagram is shown in Figure \ref{fig4}.

\begin{figure}[h]
    \centering
    \includegraphics[scale=0.7]{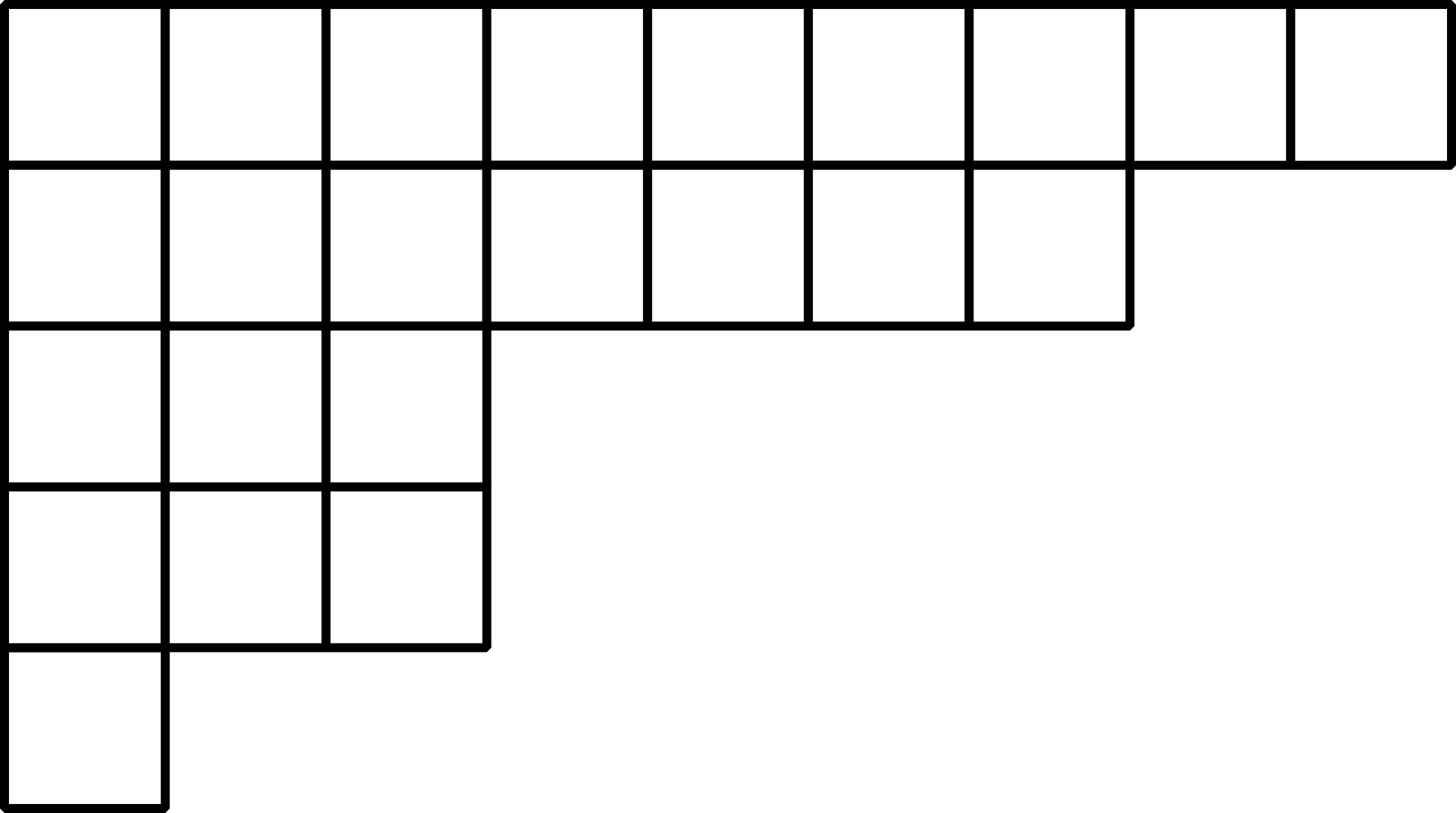}
    \caption{The Young diagram corresponding to $\lambda=[9,7,3,3,1]$.}
    \label{fig4}
\end{figure}
Recall also that each skew shape is a difference of two Young diagrams $\lambda$ and $\mu$ for which $\mu$ is inside $\lambda$. For example, if $\lambda=[9,7,3,3,1]$ and $\mu=[5,2,1]$, the skew shape corresponding to the difference $\lambda/\mu$ is shown in pink in \linebreak Figure \ref{fig5}.
\begin{figure}[h]
    \centering
    \captionsetup{justification=centering}
    \includegraphics[scale=0.7]{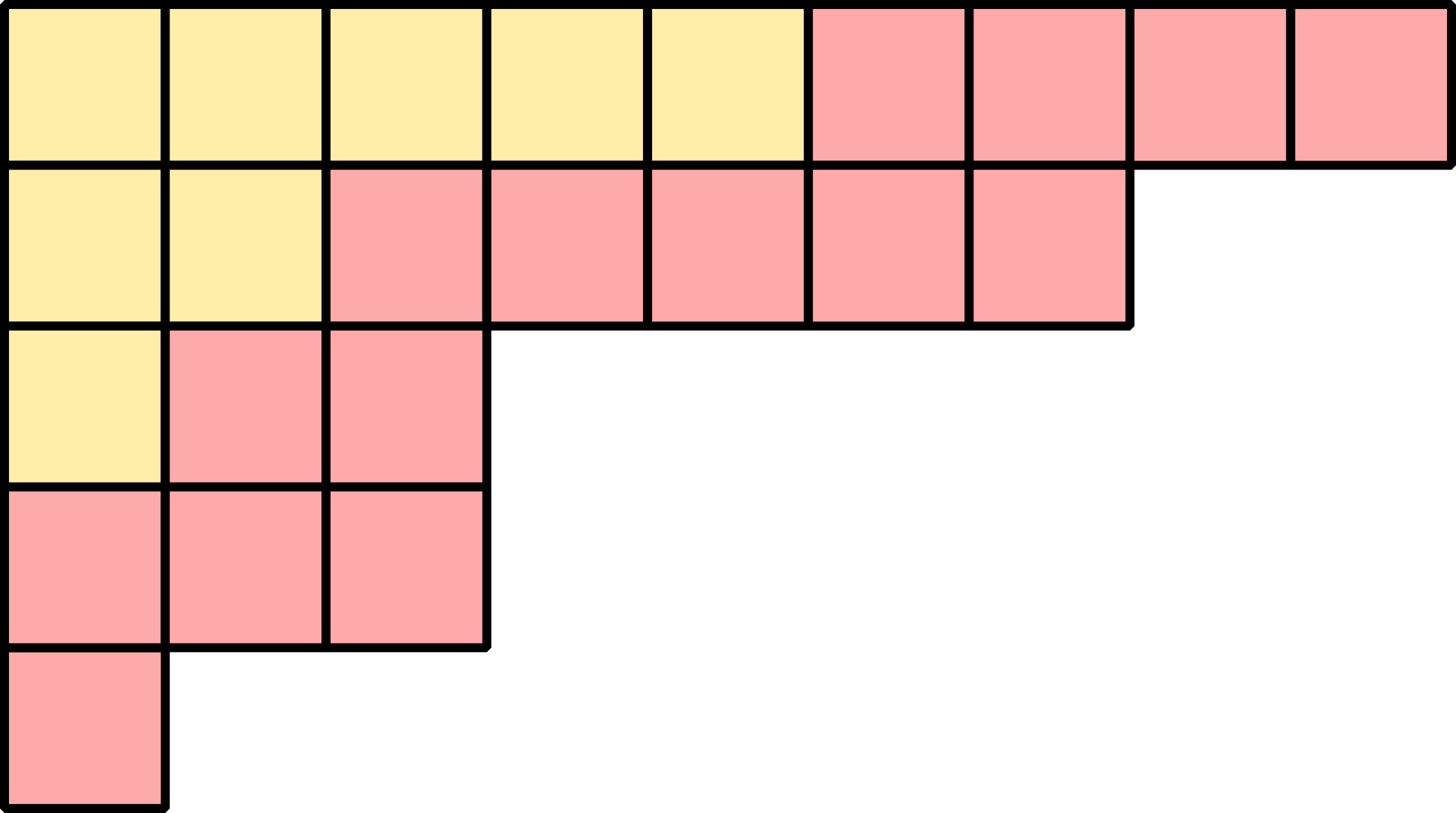}
    \caption{The skew shape corresponding to $\lambda/\mu$ for\\ $\lambda=[9,7,3,3,1]$ and $\mu=[5,2,1]$. }
    \label{fig5}
\end{figure}
\end{remark}

\section{General properties of $\sd^{m}(I,n)$}\label{gen}

In this section, we show that $\sd^{m}(I,n)$ is weakly increasing with $m$. We prove that from some $m$ onward, the function $\sd^{m}(I,n)$ stabilizes and we find the exact value of the stabilization point.

\begin{prop}\label{prop2}
For a positive integer $n > \alpha_{t} = \max(I)$, we have the following inequality
\[
\sd^{1}(I,n)\leq \sd^{2}(I,n)\leq \sd^{3}(I,n)\leq \sd^{4}(I,n)\leq\cdots.
\]
\end{prop}

\begin{proof}
By determining the first $\alpha_{t}$ elements of the permutation, we determine the entire permutation, because the rest of the elements are in weakly increasing order. We juxtapose each permutation $w\in S_{n}^{(m)}$ with a permutation $w^{\prime}\in S_{n}^{(m+1)}$, by using the first $\alpha_{t}$ elements of $w$ in $w^{\prime}$. Thus, we preserve the descent set $I$ everywhere except possibly at position $\alpha_{t}$. We want to compare $w^{\prime}_{\alpha_{t}}$ and $w^{\prime}_{\alpha_{t}+1}$ ($\alpha_{t}$-th and $(\alpha_{t}+1)$-th elements in $w^{\prime}$). In the first $\alpha_{t}$ elements of $w^{\prime}$, the number $1$ appears at most $m$ times, so there is at least one $1$ in the rest of the elements. The rest of the elements are in weakly increasing order, so $w^{\prime}_{\alpha_{t}+1}=1$. This leads us to:
\[w^{\prime}_{\alpha_{t}}=w_{\alpha_{t}}>w_{\alpha_{t}+1}\geq 1=w^{\prime}_{\alpha_{t}+1}.\]

Because we can juxtapose each permutation in $S_{n}^{(m)}$ with a different permutation in $S_{n}^{(m+1)}$, it follows that
$\sd^{m}(I,n)\leq \sd^{m+1}(I,n)$.
\end{proof}

\begin{prop}\label{prop3}
The function $\sd^{m}(I,n)$ stabilizes for $m\geq\alpha_{t}-t+1$:
\[
\sd^{\alpha_{t}-t+1}(I,n) = \sd^{\alpha_{t}-t+2}(I,n) = \sd^{\alpha_{t}-t+3}(I,n) = \cdots.
\]
\end{prop}

\begin{proof}
To determine the entire permutation, it is enough to determine the first $\alpha_{t}$ elements of the permutation, because the rest of the elements are in weakly increasing order. Let $m\geq\alpha_{t}-t+1$. The number $1$ appears at most $\alpha_{t}-t$ times in the first $\alpha_{t}$ elements, because $w_{\alpha_{i}}\neq 1$ for $i: \ 1\leq i\leq t$. So, the number $1$ appears at least once in the rest of the elements, which means that $w_{\alpha_{t}+1}=1$. Also, we can use each number from $1$ to $n$ as many times as we want in the first $\alpha_{t}$ elements, because the number $1$ appears at most $\alpha_{t}-t$ times and each number from $2$ to $n$ appears at most $\alpha_{t}-t+1$ times there. Therefore, when $m\geq \alpha_{t}-t+1$, the value of $\sd^{m}(I,n)$ is equal to the number of sequences $v=(v_{1}, v_{2}, \ldots, v_{\alpha_{t}})$, which satisfy the conditions:
\begin{itemize}
    \item $\Des(v)=I^{-}$,
    \item $v_{i}\in\{1,2,\ldots,n\}$,
    \item $v_{\alpha_{t}}\neq 1$.
\end{itemize} 

So, for $m\geq \alpha_{t}-t+1$, the value of $\sd^{m}(I,n)$ does not depend on $m$, which means that for $m\geq \alpha_{t}-t+1$, the function $\sd^{m}(I,n)$ stabilizes.  
\end{proof}

In Proposition \ref{prop3} we find that the function $\sd^{m}(I,n)$ stabilizes from some point onward. However, we do not establish the stabilization point: we only give an upper bound for it. In the next proposition, we derive the exact value of the stabilization point.  

\begin{prop}Let the set $I$ be non-empty and $L$ be the length of the longest sequence of consecutive numbers in $I$. For $n\geq L+1$, the stabilization point of the function $\sd^{m}(I,n)$ is $M=\alpha_{t}-t+1$. 
\end{prop}

\begin{proof}
To prove this it is sufficient to prove that $\sd^{\alpha_{t}-t+1}(I,n)>\sd^{\alpha_{t}-t}(I,n)$. From the proof of Proposition \ref{prop2} we know that each permutation in $S_{n}^{(\alpha_{t}-t)}$ corresponds to a permutation in $S_{n}^{(\alpha_{t}-t+1)}$, which has the same descent set. We want to show the existence of a permutation in $S_{n}^{(\alpha_{t}-t+1)}$ that has not been paired with any permutation in $S_{n}^{(\alpha_{t}-t)}$. Recall that the correspondence between the permutations in Proposition \ref{prop2} is achieved by using the same first $\alpha_{t}$ elements. So, it is enough to give an example of a permutation, which has $\alpha_{t}-t+1$ copies of some element in the first $\alpha_{t}$ elements of the permutations. To give an example of the first $\alpha_{t}$ elements of such permutation we use semistandard Young tableaux of ribbon shape (cf. Remark \ref{rem2}).

\begin{figure}[h]
    \centering
    \includegraphics[scale=0.7]{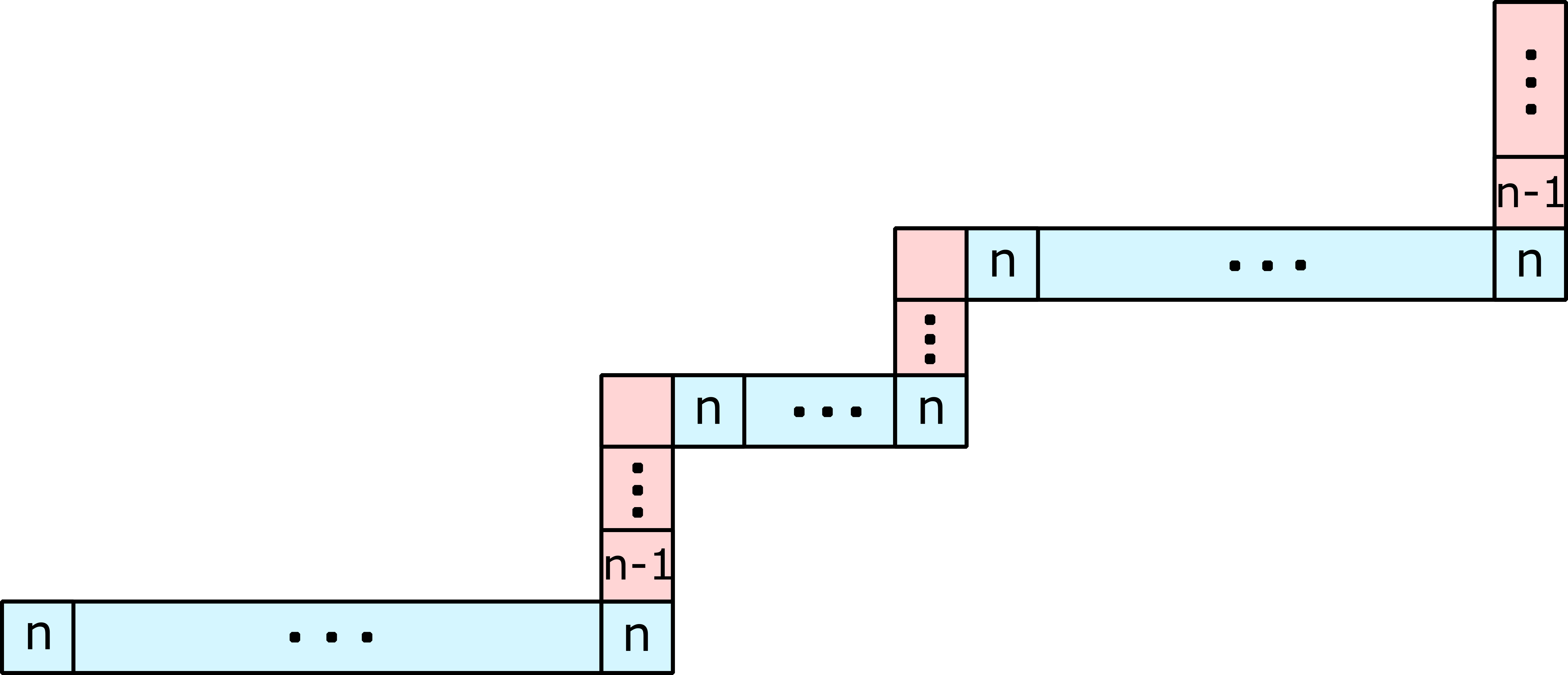}
    \caption{A permutation having $\alpha_{t}-t+1$ copies of $n$ in the first $\alpha_{t}$ elements.}
    \label{fig3}
\end{figure}

The example in Figure \ref{fig3} has elements equal to $n$ in all blue parts and in the columns, there are decreasing sequences of consecutive numbers starting from $n$. Because $n\geq L+1$, the smallest number in each column is always a positive integer. In the last column, the length of the pink part is less than $L$, which means that the last number is bigger than $1$. From Proposition $\ref{prop3}$ we know that $w_{\alpha_{t}+1}=1$, so the position $\alpha_{t}$ is in the descent set as well. We have $t-1$ squares in the pink part. Therefore, we have exactly $\alpha_{t}-t+1$ squares in the blue part, so we have exactly $\alpha_{t}-t+1$ copies of the number $n$.
\end{proof}

\section{Polynomiality}\label{pol}

In this section, we first introduce the function $D^{m}(I,n)$. We prove that $D^{m}(I,n)$ is a polynomial in $n$ for sufficiently large $n$. Then, we establish a relation between $D^{m}(I,n)$ and $\sd^{m}(I,n)$. Using this relation, we deduce that $\sd^{m}(I,n)$ is also a polynomial in $n$ for sufficiently large $n$.

\begin{definition}
For the non-empty set $I=\{\alpha_{1},\alpha_{2},\ldots,\alpha_{t}\}$, we define $D^{m}(I,n)$ to equal the number of sequences $v=(v_{1},v_{2},\ldots, v_{\alpha_{t}})$, which satisfy the conditions:
\begin{itemize}
    \item $\Des(v)=I^{-}$,
    \item $v_{i}\in\{1,2,\ldots,n\}$,
    \item each number from the set $\{1,2,\ldots,n\}$ appears at most $m$ times in $v$.
\end{itemize}
\end{definition}

\begin{definition}\label{def2}
Let the set $I=\{\alpha_{1},\alpha_{2},\ldots,\alpha_{t}\}$ be non-empty. Let $A = (a_{1}, a_{2}, \ldots, a_{r}) \in\Comp(\alpha_{t},m)$ be a composition. By $C(A,I)$, we denote the number of sequences $v=(v_{1}, v_{2}, \ldots, v_{\alpha_{t}})$, which satisfy the following conditions:
\begin{itemize}
    \item $\Des(v)=I^{-}$,
    \item $v_{i}\in\{1,2,\ldots,r\}$,
    \item for each $j\in\{1,2,\ldots,r\}$, the number $j$ appears exactly $a_{j}$ times in $v$.
\end{itemize}
\end{definition}

\begin{remark}\label{rem1}
Notice that the value of $C(A,I)$ does not change if we replace the set $\{1,2,3,\dots,r\}$, from which we choose the elements of $v$,  with another set of $r$ elements as long as the smallest number in the set appears $a_{1}$ times, the next in size appears  $a_{2}$ times, and so on.
\end{remark}

\begin{lemma}\label{lem2}
For $n\geq\alpha_{t}$, the function $D^{m}(I,n)$ is a polynomial in $n$.
\end{lemma}

\begin{proof}
From Definition \ref{def2}, we know that 
\[
D^{m}(I,n)=\sum\limits_{A\in \Comp(\alpha_{t},m)} C(A,I)\binom{n}{r},
\]
where $r$ is the length of the composition $A$ and we sum over all compositions in $\Comp(\alpha_{t},m)$.

We multiply $C(A,I)$ by $\binom{n}{r}$, because as we noted in Remark \ref{rem1}, the set $\{1,2,3,\dots,r\}$ can be replaced with any set with $r$ elements. Because we have the numbers from $1$ to $n$, there are $\binom{n}{r}$ ways to choose $r$ of them to form a set.

Let $n\geq\alpha_{t}$. Because $r$ is the length of the composition $A\in\Comp(\alpha_{t},m)$, it follows that $\alpha_{t}\geq r$. We obtain  $n\geq\alpha_{t}\geq r$, which means that $\binom{n}{r}$ is a polynomial in $n$. Also, $C(A,I)$ and the number of elements in the sum are both determined only by $m$ and $I$. Therefore, $D^{m}(I,n)$ is polynomial in $n$ for $n\geq\alpha_{t}$. 
\end{proof}

\begin{theorem}\label{theo5}
For any positive integer $m$, the function $\sd^{m}(I,n)$ is a polynomial in $n$, for $n\geq\alpha_{t}$.
\end{theorem}

\begin{proof}

To prove this we use induction on $t$ --- the size of $I$.

When $t=0$, we get that $I=\varnothing$ and $\sd^{m}(n,\varnothing)=1$.

Let $w=w_{1}w_{2}\ldots w_{nm}$ be a permutation in $S_{n}^{(m)}$. By determining the first $\alpha_{t}$ elements of the permutation, we determine the entire permutation, because the rest of the elements are in weakly increasing order. There are $D^{m}(I,n)$ ways to define the first $\alpha_{t}$ elements. However, by defining the first $\alpha_{t}$ elements we do not know what is happening between $w_{\alpha_{t}}$ and $w_{\alpha_{t}+1}$. If $w_{\alpha_{t}}>w_{\alpha_{t}+1}$ we get that $\Des(w)=I$. If $w_{\alpha_{t}}\leq w_{\alpha_{t}+1}$ we get that $\Des(w)=I^{-}$. So, we obtain the equation
$$D^{m}(I,n)=\sd^{m}(I,n)+\sd^{m}(I^{-},n).$$

From the induction hypothesis we get that $\sd^{m}(I^{-},n)$ is a polynomial in $n$ for $n\geq\alpha_{t-1}$. Using Lemma \ref{lem2} we get that $D^{m}(I,n)$ is a polynomial in $n$ for $n\geq\alpha_{t}$. Therefore, $\sd^{m}(I,n)$ is also a polynomial in $n$ for $n\geq\alpha_{t}$, which finishes the induction.
\end{proof}

\section{Polynomiality using the Jacobi-Trudi identity}\label{jac}

In this section, we give an alternative proof for the fact that $\sd^{m}(I,n)$ is a polynomial in $n$ for sufficiently large $n$. To do this we use the Jacobi-Trudi identity. This proof is easier and more elegant, but we need some background in symmetric functions to understand it. For this reason, we first start with a brief introduction to this field.

\subsection{Symmetric functions}

Here we introduce the theory we need for using the Jacobi-Trudi identity. To do this, we follow Stanley's  book \textit{Enumerative Combinatorics, vol. 2} \cite{stanley1986enumerative}.

\begin{definition}
Let $\lambda=[\lambda_{1},\lambda_{2},\ldots,\lambda_{s},0,0,\ldots]$ be a partition of a positive integer $t$. We define the \textit{Monomial symmetric function}
\[
m_{\lambda}(x_{1},x_{2},x_{3}\ldots)=\sum\limits_{\alpha}x^{\alpha}=\sum\limits_{\alpha}\prod\limits_{i=1}^{\infty}x_{i}^{\alpha_{i}}, 
\]
where the sum is over all distinct permutations $\alpha=(\alpha_{1},\alpha_{2},\alpha_{3},\ldots)$ of $\lambda$. 
\end{definition}

\begin{definition}
Let $t$ be a positive integer. We define the \textit{Complete homogeneous symmetric function} 
\[
h_{t}(x_{1},x_{2},x_{3}\ldots)=\sum\limits_{i_{1}\leq i_{2}\leq\cdots\leq i_{t}} x_{i_{1}} x_{i_{2}}\cdots x_{i_{t}} = \sum\limits_{\lambda\vdash t} m_{\lambda},
\]
where the sum is over all partitions $\lambda$ of $t$.
\end{definition}

\begin{definition}
Let $\lambda/\mu$ be a skew shape. We define the \textit{skew Schur function}
\[
s_{\lambda/\mu}(x_{1},x_{2},x_{3}\ldots) = \sum\limits_{T} x^{T} = \sum\limits_{T}\prod\limits_{i=1}^{\infty}x_{i}^{t_{i}}, 
\]
where the sum is over all semistandard Young tableaux $T$ of shape $\lambda/\mu$ and $t_{i}$ is equal to the occurrences of the number $i$ in $T$. 
\end{definition}

\begin{remark}\label{rem1}
Recall from Remark $\ref{rem2}$ that the permutations we are considering can be visualized by semistandard Young tableaux of ribbon shape. Ribbon shapes are skew shapes, so they can be represented as a difference $\lambda/\mu$ of two partitions. Therefore, $\sd^{m}(I,n)$ is equal to the coefficient in front of $x_{1}^{m}x_{2}^{m}\cdots x_{n}^m$ in  $s_{\lambda/\mu}$.
\end{remark}

\subsection{The Jacobi-Trudi identity proof}
Here we first introduce the Jacobi-Trudi identity. Then we use it to give a second proof that for any $m$, the function $\sd^{m}(I,n)$ is a polynomial in $n$ for sufficiently large $n$.

\begin{theorem}[Jacobi-Trudi identity, Theorem 7.6.1 \cite{stanley1986enumerative}]\label{theo3}
Let $\lambda/\mu$ be a skew shape and let $k$ be the number of non-zero elements in $\lambda$. Then
\[
s_{\lambda/\mu}=\det \left[h_{\lambda_{i}-\mu_{j}-i+j}\right]_{i,j=1}^{k},
\]
where $h_{i}=0$ for $i<0$ and $h_{0}=1$.
\end{theorem}

We do not give proof of the Jacobi-Trudi identity because we want to keep this paper concise. A proof of it can be seen in Richard P. Stanley's  book \textit{Enumerative Combinatorics} \cite{stanley1986enumerative}.

\begin{theorem}\label{theo4}
For any $m$, the function $\sd^{m}(I,n)$ is a polynomial in $n$ for sufficiently large $n$.
\end{theorem}

\begin{proof}
From Remark \ref{rem1} we know that the permutations we are considering correspond to semistandard Young tableaux of shape $\lambda/\mu$ and it is sufficient to prove that the coefficient in front of $x_{1}^{m}x_{2}^{m}\cdots x_{n}^m$ in  $s_{\lambda/\mu}$ is a polynomial in $n$. To do that we use the Jacobi-Trudi identity (Theorem \ref{theo3}). 
\begin{figure}[h]
    \centering
    \includegraphics[scale=0.70]{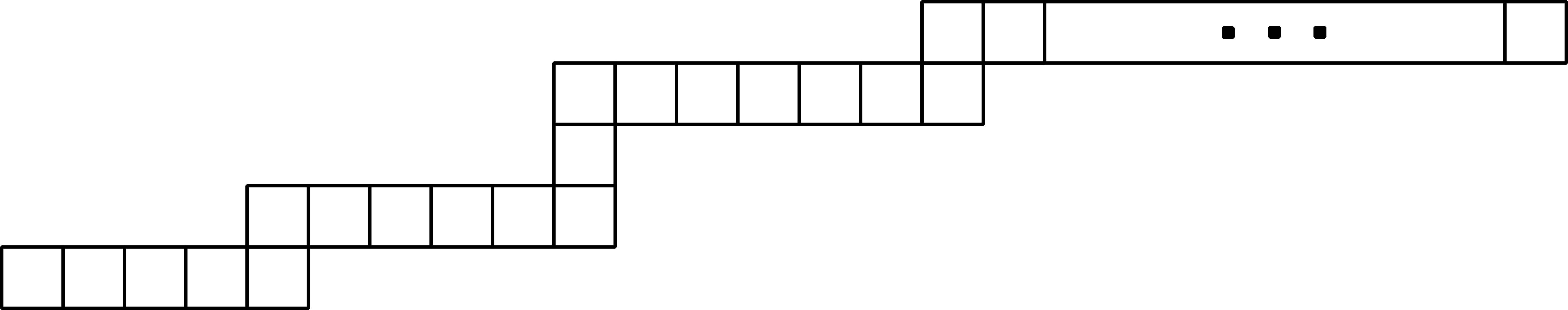}
    \caption{\centering
    Sketch of the shape $\lambda/\mu$ that corresponds to permutations we are considering in the proof of Theorem \ref{theo4}.}
    \label{fig6}
\end{figure}

Figure \ref{fig6} shows a sketch of the skew shape $\lambda/\mu$, which we are considering. The partition $\mu$ and all the elements in the partition $\lambda$ except $\lambda_{1}$ are determined by the set $I$. The number $k$ of non-zero elements in $\lambda$ is also determined by $I$. The number of elements in the first row of our skew shape is $nm-\alpha_{t}=\lambda_{1}-\mu_{1}$, from which we derive $\lambda_{1}=nm-\alpha_{t}+\mu_{1}$. Therefore,
each non-zero term in the determinant is equal to
\[
h_{nm-u}h_{u_{1}}h_{u_{2}}h_{u_{3}}\cdots h_{u_{k-1}}
\]
for some non-negative integers $u,u_{1},u_{2},u_{3}\ldots,u_{k-1}$ for which 
\[
u=u_{1}+u_{2}+u_{3}+\cdots+u_{k-1}.
\] 

It is important to mention that $u,u_{1},u_{2},u_{3}\ldots,u_{k-1}$ are determined only by the descent set $I$. Let us look closely at each term in the determinant:
\[
h_{nm-u}h_{u_{1}}h_{u_{2}}h_{u_{3}}\cdots h_{u_{k-1}}=
h_{nm-u}\sum\limits_{\nu\vdash u} v(\nu) m_{\nu}=
\sum\limits_{\nu\vdash u} v(\nu) m_{\nu} h_{nm-u}.
\]

The coefficient $v(\nu)$ is determined by the partition $\nu$ and the coefficients $u_{1},u_{2},u_{3}\ldots,u_{k-1}$. We want to calculate the coefficient in front of $x_{1}^{m}x_{2}^{m}\cdots x_{n}^m$ in $m_{\nu} h_{nm-u}$. Let the partition $\nu$ be equal to 
\[
\nu=[\underbrace{\nu_{1},\ldots,\nu_{1}}_{r_{1}},\underbrace{\nu_{2},\ldots,\nu_{2}}_{r_{2}},\ldots, \underbrace{\nu_{\ell},\ldots,\nu_{\ell}}_{r_{\ell}}],
\]
where $r_{i}$ is the multiplicity of $\nu_{i}$.

If $\nu_{1}>m$, the coefficient in front of $x_{1}^{m}x_{2}^{m}\cdots x_{n}^m$ in $m_{\nu} h_{nm-u}$ is $0$. 

If $\nu_{1}\leq m$, the coefficient in front of $x_{1}^{m}x_{2}^{m}\cdots x_{n}^m$ in $m_{\nu} h_{nm-u}$ is
\begin{align*}
\binom{n}{r_{1}} \binom{n-r_{1}}{r_{2}} \binom{n-r_{1}-r_{2}}{r_{3}}\cdots \binom{n-r_{1}-r_{2}-\cdots-r_{\ell-1}}{r_{\ell}}=\\
\binom{n}{r_{1}+r_{2}+\cdots+r_{\ell}} \frac{(r_{1}+r_{2}+\cdots+r_{\ell})!}{r_{1}!r_{2}!\cdots r_{\ell}!}=
\binom{n}{u} \frac{u!}{r_{1}!r_{2}!\cdots r_{\ell}!}.
\end{align*}

The coefficients $u,r_{1},r_{2},\ldots,r_{\ell}$ are determined only by $I$ and $m$, so if $n\geq u$, we get that the coefficient in front of $x_{1}^{m}x_{2}^{m}\cdots x_{n}^m$ in $m_{\nu} h_{nm-u}$ is a polynomial in $n$. The sum in each term of the determinant is over all partitions of $u$, so the number of summands in the sum is fixed. Therefore, the coefficient in front of $x_{1}^{m}x_{2}^{m}\cdots x_{n}^m$ in each term of the determinant is a polynomial in $n$. The set $I$ is fixed, thus we know that the number $k$ of non-zero elements in $\lambda$ is fixed too. So, there are a fixed number of elements in the determinant in the Jacobi-Trudi identity. Therefore, the coefficient in front of $x_{1}^{m}x_{2}^{m}\cdots x_{n}^m$ in $s_{\lambda/\mu}$ is a polynomial in $n$ for $n\geq u$, where $u$ is determined only by $I$.
\end{proof}

\section{Formula for $d^{\infty}(I,n)$}\label{form}

In Section \ref{gen}, we proved that $\sd^{m}(I,n)$ stabilizes for  $m\geq\alpha_{t}-t+1$ and in Section \ref{pol}, we proved that $\sd^{m}(I,n)$ is a polynomial in $n$ for $n\geq\alpha_{t}$. Therefore, it is reasonable to introduce a new notation for the stabilized polynomial when $n,m\geq\alpha_{t}$. 

\begin{definition}
Let $d^{\infty}(I,n)$ be the polynomial that equals the function $\sd^{m}(I,n)$ for $n,m\geq\alpha_{t}$.
\end{definition}

In this section, we derive an explicit formula for $d^{\infty}(I,n)$. However, we first need to introduce some crucial notations. 

\begin{definition}
For the set $I=\{\alpha_{1},\alpha_{2},\ldots,\alpha_{t}\}$, we let $\beta = (\beta_{1},\beta_{2},\ldots,\beta_{t})$ denote the sequence of the \textit{first differences} of the sequence $(0,\alpha_{1}, \alpha_{2},\ldots,\alpha_{t})$. So, \ $\beta_{1}=\alpha_{1}-0$, \ $\beta_{2}=\alpha_{2}-\alpha_{1}, \ldots, \ \beta_{t}=\alpha_{t}-\alpha_{t-1}$. 
\end{definition}

When we talk about compositions we imagine putting separators between balls in a line. For example, the division of the balls below corresponds to the composition $(3,1,2,2)$.
$$. \ . \ . \ | \ . \ | \ . \ . \ | \ . \ .$$

As we can see the numbers of balls between the separators give us the elements of the composition. Let us put a weight on each ball and instead of taking the number of balls, we take the total weight of the balls between the separators. Let us look at the previous example, but this time put weights on the balls. We take the weights to be $\beta_{1},\beta_{2},\beta_{3},\beta_{4},\beta_{5},\beta_{6},\beta_{7},\beta_{8}.$ Now the composition $(3,1,2,2)$ corresponds to
$$\beta_{1} \ \beta_{2} \ \beta_{3} \ | \ \beta_{4} \ | \ \beta_{5} \ \beta_{6} \ | \ \beta_{7} \ \beta_{8}.$$
We can write this in short as  ($\sigma_{1},\sigma_{2},\sigma_{3},\sigma_{4}$), where $\sigma_{1}=\beta_{1}+\beta_{2}+\beta_{3}$, \linebreak $\sigma_{2}=\beta_{4}, \ \sigma_{3}=\beta_{5}+\beta_{6}$, and $\sigma_{4}=\beta_{7}+\beta_{8}$. In the next definition, we define this more generally.

\begin{definition}\label{def1}
Let $A\in\Comp(t)$ be a composition of $t$ and $A=(a_{1},a_{2},\dots,a_{s})$. For the sequence $\beta=(\beta_{1},\beta_{2},\ldots,\beta_{t})$, we define the function $f_{\beta}$, such that $f_{\beta}(A)=(\sigma_{1},\sigma_{2},\dots,\sigma_{s})$, where 
\[
\sigma_{1}=\beta_{1}+\beta_{2}+\cdots+\beta_{a_{1}}
\]
\[
\sigma_{i}=\beta_{a_{1}+a_{2}+\cdots+a_{i-1}+1}+\beta_{a_{1}+a_{2}+\cdots+a_{i-1}+2}+\cdots+\beta_{a_{1}+a_{2}+\cdots+a_{i}} \text{ \ for all } i: \ 2\leq i\leq s.
\]
\end{definition}

To derive an explicit formula for $d^{\infty}(I,n)$ we first need to prove the following lemma.

\begin{lemma}\label{lem3}
Let $j$ and $n$ be positive integers such that $j \leq n$. For the non-empty set $I = \{\alpha_1, \alpha_2, \ldots, \alpha_t\}$, we define $\beta = (\beta_1, \beta_2, \ldots, \beta_t)$ as in Section \ref{pre}. Then, the number of sequences $v = (v_{1}, v_{2}, \ldots, v_{\alpha_t})$, which satisfy the conditions:
\begin{itemize}
    \item $v_{i}\in \{1,2,\ldots,n\}$ 
    \item $v_{\alpha_{t}}=j$,
    \item $\Des(v)=I^{-}$.
\end{itemize}
is
\begin{equation}\label{equ2}
\sum\limits_{A\in \Comp(t)}(-1)^{t-s} \binom{n-1+\sigma_{1}}{\sigma_{1}}\cdots\binom{n-1+\sigma_{s-1}}{\sigma_{s-1}}\binom{j-1+\sigma_{s}-1}{\sigma_{s}-1},
\end{equation}
where the summation is over all compositions $A \in \Comp(t)$. Recall that from \linebreak Definition \ref{def1} the variables $\sigma_1, \sigma_2, \ldots, \sigma_s$ depend on $A$ by the relation \linebreak $f_{\beta}(A) = (\sigma_1, \sigma_2, \ldots, \sigma_s)$.
\end{lemma}

\begin{proof}
To prove this statement we make an induction on $t$ --- the number of elements in $I$.

When $t=1$, we have that $I=\alpha_{1}$ and $I^{-}=\varnothing$. Therefore, the sequence $v$ should look like 
\[
1\leq v_{1}\leq v_{2}\leq v_{3}\leq\dots\leq v_{\alpha_{1}}=j.
\]

There are $\binom{j-1+\alpha_{1}-1}{\alpha_{1}-1}$ such sequences. Let us see what result we get by applying formula \eqref{equ2}. When $t=1$, we get that  $\beta=(\alpha_{1})$ and the set $\Comp(1)=\{(1)\}$. Therefore, $f_{\beta}((1))=\alpha_{1}$ and $\sigma_{1}=\alpha_{1}$. Substituting this in formula \eqref{equ2}, we get
\begin{align*}
\sum\limits_{A\in \Comp(t)}(-1)^{t-s}\binom{n-1+\sigma_{1}}{\sigma_{1}}  \dots \binom{n-1+\sigma_{s-1}}{\sigma_{s-1}} \binom{j-1+\sigma_{s}-1}{\sigma_{s}-1}=\\
=(-1)^{1-1} \binom{j-1+\sigma_{1}-1}{\sigma_{1}-1} = \binom{j-1+\alpha_{1}-1}{\alpha_{1}-1},
\end{align*}
so formula \eqref{equ2} is true for $t=1$. Let us suppose it is true for $t-1$. We want to prove that it is true for $t$. We fix $v_{\alpha_{t-1}}=i$. From the induction hypothesis, we know that there are
\[
\sum\limits_{A\in \Comp(t-1)}(-1)^{t-s-1}\binom{n-1+\sigma_{1}}{\sigma_{1}}  \dots \binom{n-1+\sigma_{s-1}}{\sigma_{s-1}} \binom{i-1+\sigma_{s}-1}{\sigma_{s}-1}
\]
ways to choose the first $\alpha_{t-1}$ elements of the sequence. Let us look at the last $\alpha_{t}-\alpha_{t-1}$ elements of the sequence:
\[
i > v_{\alpha_{t-1}+1}\leq v_{\alpha_{t-1}+2} \leq \cdots \leq v_{\alpha_{t}-1} \leq v_{\alpha_{t}}=j.
\]

If $i+1\geq j$, there are
\[
\binom{j-1+\alpha_{t}-\alpha_{t-1}-1}{\alpha_{t}-\alpha_{t-1}-1}=
\binom{j-1+\beta_{t}-1}{\beta_{t}-1}
\]
ways to choose the last $\alpha_{t}-\alpha_{t-1}$ elements of the sequence.

If $i\leq j$, there are
\[
\binom{j-1+\beta_{t}-1}{\beta_{t}-1}-\binom{j-i+\beta_{t}-1}{\beta_{t+1}-1} 
\]
ways to choose the last $\alpha_{t}-\alpha_{t-1}$ elements of the sequence.

Therefore, the number of sequences $v$ with a descent set $I^{-}=\{\alpha_{1},\alpha_{2},\dots,\alpha_{t-1}\}$, last element $v_{\alpha_{t}}=j$ and elements from the set $\{1,2,\ldots,n\}$ is

$\sum\limits_{i=1}^{j} \left(\binom{j-1+\beta_{t}-1}{\beta_{t}-1}-\binom{j-i+\beta_{t}-1}{\beta_{t}-1}\right) \sum\limits_{A\in \Comp(t-1)}(-1)^{t-s-1}\binom{n-1+\sigma_{1}}{\sigma_{1}}  \cdots \binom{n-1+\sigma_{s-1}}{\sigma_{s-1}} \binom{i-1+\sigma_{s}-1}{\sigma_{s}-1}$

$\hphantom{=}+\sum\limits_{i=j+1}^{n} \binom{j-1+\beta_{t}-1}{\beta_{t}-1} \sum\limits_{A\in \Comp(t-1)}(-1)^{t-s-1}\binom{n-1+\sigma_{1}}{\sigma_{1}}  \cdots \binom{n-1+\sigma_{s-1}}{\sigma_{s-1}} \binom{i-1+\sigma_{s}-1}{\sigma_{s}-1}$

$=\sum\limits_{i=1}^{n} \sum\limits_{A\in \Comp(t-1)}(-1)^{t-s-1}\binom{n-1+\sigma_{1}}{\sigma_{1}}  \cdots \binom{n-1+\sigma_{s-1}}{\sigma_{s-1}} \binom{i-1+\sigma_{s}-1}{\sigma_{s}-1} \binom{j-1+\beta_{t}-1}{\beta_{t}-1} $ 

$\hphantom{=}+\sum\limits_{i=1}^{j} \sum\limits_{A\in \Comp(t-1)}(-1)^{t-s}\binom{n-1+\sigma_{1}}{\sigma_{1}}  \cdots \binom{n-1+\sigma_{s-1}}{\sigma_{s-1}} \binom{i-1+\sigma_{s}-1}{\sigma_{s}-1} \binom{j-i+\beta_{t}-1}{\beta_{t}-1}$

$=\sum\limits_{A\in \Comp(t-1)}(-1)^{t-(s+1)}\binom{n-1+\sigma_{1}}{\sigma_{1}}  \cdots \binom{n-1+\sigma_{s-1}}{\sigma_{s-1}} \binom{n-1+\sigma_{s}}{\sigma_{s}}\binom{j-1+\beta_{t}-1}{\beta_{t}-1}$

$\hphantom{=}+\sum\limits_{A\in \Comp(t-1)}(-1)^{t-s}\binom{n-1+\sigma_{1}}{\sigma_{1}}  \cdots \binom{n-1+\sigma_{s-1}}{\sigma_{s-1}} \binom{j-1+\sigma_{s}+\beta_{t}-1}{\sigma_{s}+\beta_{t}-1}$

$=\sum\limits_{A\in \Comp(t)}(-1)^{t-s}\binom{n-1+\sigma_{1}}{\sigma_{1}}  \cdots \binom{n-1+\sigma_{s-1}}{\sigma_{s-1}} \binom{j-1+\sigma_{s}-1}{\sigma_{s}-1},$\\
which finishes the induction.
\end{proof}

In the next theorem, we derive an explicit formula for $d^{\infty}(I,n)$. To do this we use the result from the previous lemma.

\begin{theorem}\label{theo6}
For the non-empty set $I = \{\alpha_1, \alpha_2, \ldots, \alpha_t\}$, we define $\beta = (\beta_1, \beta_2, \ldots, \beta_t)$ as in Section \ref{pre}. Then: 
\begin{equation*}\large{
d^{\infty}(I,n)=\sum\limits_{A\in \Comp(t)}(-1)^{t-s}\binom{n-1+\sigma_{1}}{\sigma_{1}} \dots \binom{n-1+\sigma_{s-1}}{\sigma_{s-1}}
\left(\binom{n-1+\sigma_{s}}{\sigma_{s}}-1\right) }, 
\end{equation*}
where, in the summation above, we sum over all compositions $A \in \Comp(t)$. Recall that from Definition \ref{def1} the variables $\sigma_{1}, \sigma_{2}, \ldots, \sigma_{s}$ depend on $A$ by the relation $f_{\beta}(A) = (\sigma_1, \sigma_2, \ldots, \sigma_s)$.
\end{theorem}

\begin{proof}
From the proof of Proposition \ref{prop3}, we know that $d^{\infty}(I,n)$ is equal to the number of sequences $v=(v_{1},v_{2},v_{3},\ldots ,v_{\alpha_{t}})$, which satisfy the conditions:
\begin{itemize}
    \item $\Des(v)=I^{-}$,
    \item $v_{i}\in\{1,2,\ldots,n\}$,
    \item $v_{\alpha_{t}}\neq 1$.
\end{itemize} 

We can compute the number of sequences $v$ by summing over all $j$ from $2$ to $n$ in Lemma \ref{lem3}. Therefore, we obtain that
\begin{equation*}
\begin{split}
d^{\infty}(I,n) & = \sum\limits_{j=2}^{n} \sum\limits_{A\in \Comp(t)}(-1)^{t-s}\binom{n-1+\sigma_{1}}{\sigma_{1}} \dots \binom{n-1+\sigma_{s-1}}{\sigma_{s-1}}
\binom{j-1+\sigma_{s}-1}{\sigma_{s}-1} \\
& =\sum\limits_{A\in \Comp(t)}(-1)^{t-s}
\binom{n-1+\sigma_{1}}{\sigma_{1}} \dots \binom{n-1+\sigma_{s-1}}{\sigma_{s-1}}
\left(\binom{n-1+\sigma_{s}}{\sigma_{s}}-1\right).    
\end{split}
\end{equation*}
\end{proof}

\section{Coefficients for $d^{\infty}(I,n)$}\label{coef}

In this section, we look at the coefficients of $d^{\infty}(I,n)$ in bases of the type $\left(\binom{n+k}{i}\right)_{i=0}^{\infty}$. We determine when these coefficients are positive and derive the exact values for some of them.

\begin{definition}\label{def5}
For a non-empty set $I$ and an integer $i$, we define $d_{i}^{\infty}(I,n)$ to equal the number of sequences $v=(v_{1},v_{2},\ldots,v_{\alpha_{t}})$, which satisfy the following conditions:
\begin{itemize}
\item $\Des(v)=I^{-}$,
\item $v_{j}\in\{1,2,\dots n\}$,
\item exactly $i$ numbers from the set
$\{2,3,\dots n\}$ appear in $v$,
\item $v_{\alpha_{t}}\neq 1$.
\end{itemize}
\end{definition}

\begin{definition}\label{def3}
Let $b_{0}(I),b_{1}(I),b_{2}(I),\ldots$ be the coefficients of $d^{\infty}(I,n)$ in the base  
$\left(\binom{n-1}{i}\right)_{i=0}^{\infty}$:
\[
d^{\infty}(I,n) = b_{0}\binom{n-1}{0} + b_{1}\binom{n-1}{1} + b_{2}\binom{n-1}{2}+\cdots.
\]
\end{definition}

In the next theorem, we prove some interesting properties of the coefficients $b_{i}(I)$. We use ideas similar to the ideas that Diaz-Lopez et al.\cite{diaz2019descent} used in the proof of Theorem 3.3 in their paper \textit{Descent polynomials}.

\begin{theorem}\label{theo1}
Let the set $I=\{\alpha_{1},\ldots,\alpha_{t}\}$ be non-empty and $L$ be the length of the longest sequence of consecutive numbers in $I$.  For each $i$ such that \linebreak $L\leq i\leq\alpha_{t}$, the coefficient $b_{i}(I)$ is a positive integer. For $i$ such that $i<L$ or $i>\alpha_{t}$, the coefficient $b_{i}(I)=0$.
\end{theorem}

\begin{proof}
To prove this statement, we first want to find a combinatorial interpretation of the coefficients: $b_{0}(I),b_{1}(I),b_{2}(I),\ldots.$    

In the proof of Proposition \ref{prop3} we conclude that the value of $d^{\infty}(I,n)$ is equal to the number of sequences $v=(v_{1},v_{2},\ldots,v_{\alpha_{t}})$, which satisfy the following conditions:
\begin{itemize}
\item $\Des(v)=I^{-}$,
\item $v_{j}\in\{1,2,\dots n\}$,
\item $v_{\alpha_{t}}\neq 1$.
\end{itemize}

Therefore, from Definition \ref{def5}, we obtain that
\begin{equation}\label{equ3}
d^{\infty}(I,n)= d_{0}^{\infty}(I,n)+d_{1}^{\infty}(I,n)+d_{2}^{\infty}(I,n)+ d_{3}^{\infty}(I,n)+\cdots.    
\end{equation}

Let $v$ be a sequence that contains exactly $i$ of the elements from the set $\{2,3,4,\ldots,n\}$. Let us replace these $i$ elements with other $i$ elements from the set $\{2,3,4,\ldots,n\}$, such that the $j$-th number in size from the old $i$ elements is replaced with the $j$-th number in size from the new $i$ elements. By doing that, we keep the descent set $I^{-}$ and the last element remains bigger than $1$. So, $d_{i}^{\infty}(I,n)$ is equal to $\binom{n-1}{i}$ times the number of sequences $v=(v_{1},v_{2},\ldots,v_{\alpha_{t}})$, which satisfy the conditions:
\begin{itemize}
\item $\Des(v)=I^{-}$,
\item $v_{j}\in\{1,2,\dots i+1\}$,
\item for every $\ell\in\{2,3,\dots i+1\}$, there exists $j$ such that $v_{j}=\ell$,
\item $v_{\alpha_{t}}\neq 1$.
\end{itemize}

If we write $d_{i}^{\infty}(I,n)$ as $\binom{n-1}{i}$ times the number of these sequences in Equation \eqref{equ3}, we get a representation of $d^{\infty}(I,n)$ in the base $\left(\binom{n-1}{i}\right)_{i=0}^{\infty}$. Therefore, the coefficient $b_{i}(I)$ is equal to the number of sequences that satisfy the same four conditions we just mentioned above:
\begin{itemize}
\item $\Des(v)=I^{-}$,
\item $v_{j}\in\{1,2,\dots i+1\}$,
\item for every $\ell\in\{2,3,\dots i+1\}$, there exists $j$ such that $v_{j}=\ell$,
\item $v_{\alpha_{t}}\neq 1$.
\end{itemize}

Let $v$ be one such sequence. Because the length of $v$ is $\alpha_{t}$, we get that there at most $\alpha_{t}$ different numbers in $v$. Also, from the first condition, we get that there are at least $L$ different numbers in $v$ bigger than $1$.  However, from the third condition we get that there are exactly $i$ different numbers in $v$ bigger than $1$, which leads us to the inequality:
\[L\leq\ i \leq\alpha_{t}.\]

From this inequality, we conclude that:
\[0=b_{0}(I) = b_{1}(I) = \cdots = b_{L-1}(I) = b_{\alpha_{t}+1}(I) = b_{\alpha_{t}+2}(I) = \cdots.\]

Now, we must prove that $b_{L}(I),b_{L+1}(I),\ldots, b_{\alpha_{t}}(I)$ are positive integers. For each $i$ such that $L\leq i \leq\alpha_{t}$, we need to show an example of a sequence counted by $b_{i}(I)$. To show such examples we use semistandard Young tableaux of ribbon shape (cf. Remark \ref{rem2}).

First, we show a sequence counted by $b_{L}(I)$. We let all numbers in positions that are not in the descent set $I$ to be equal to $1$. In the columns of the Young tableau, we put consecutive numbers in increasing order starting from $1$. In the example in Figure \ref{fig1}, we show one such sequence. 

\begin{figure}[h]
\centering
  \centering
  \includegraphics[scale=0.7]{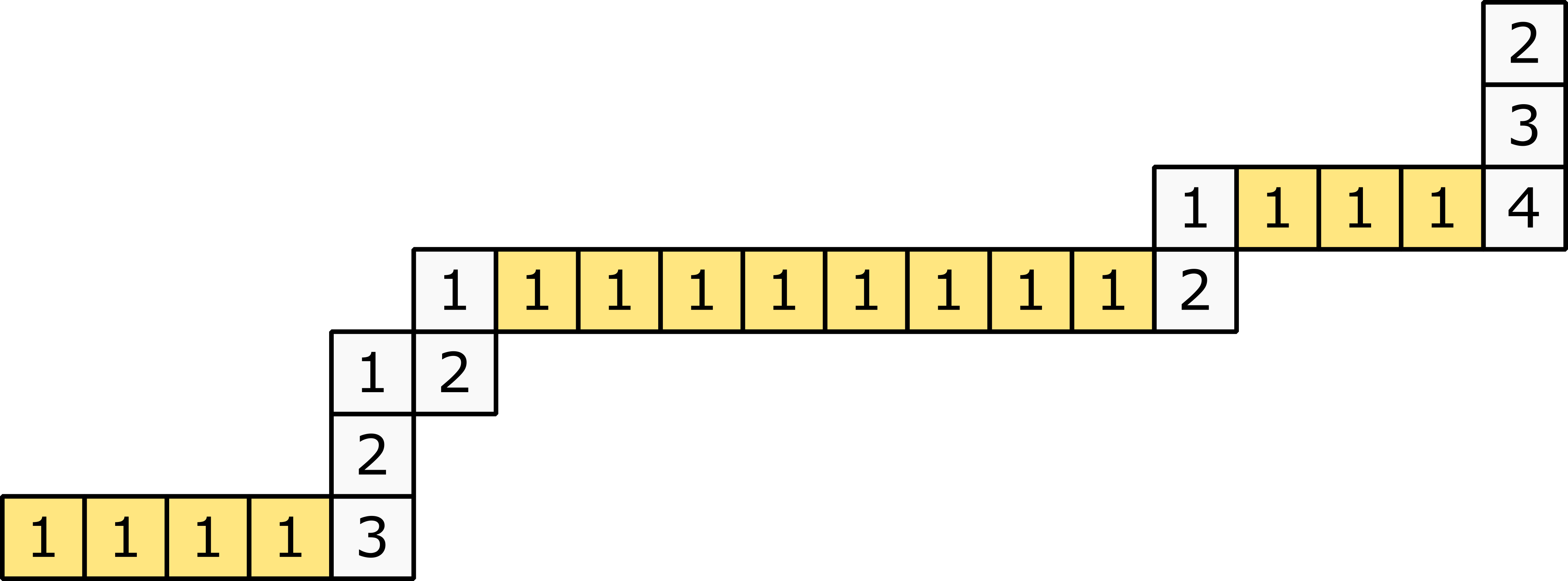}
  \caption{An example of a sequence counted by $b_{L}(I)$.}
  \label{fig1}
\end{figure}

Then, we want show a sequence counted by $b_{\alpha_{t}}(I)$. To ease the explanation we define two more sets related to the descent set $I$:
\[I^{\prime}=\{i : \ i\in I, \ i-1\notin I\}, \ I^{\prime}=\{\alpha_{1}^{\prime},\alpha_{2}^{\prime},\ldots,\alpha_{p}^{\prime}\},\]
\[I^{''}=\{i : \ i\notin I, \ i-1\in I\}, \ I^{''}=\{\alpha_{1}^{''},\alpha_{2}^{''},\ldots,\alpha_{p}^{''}\}. \]
It is worth mentioning that
\[\alpha_{1}^{\prime} < \alpha_{1}^{''} < \cdots < \alpha_{p}^{\prime} < \alpha_{p}^{''} \text{ \ and \ } \alpha_{p}^{''}=\alpha_{t}+1.\] 

Now, we show formally a sequence counted by $b_{\alpha_{t}}(I)$. On the left we have the positions in the sequence and on the right we have the numbers corresponding to these positions:  

\[(1,2,\ldots,\alpha_{1}^{\prime}-1) \longrightarrow (2,3,\ldots,\alpha_{1}^{\prime})\]
\[(\alpha_{1}^{''},\alpha_{1}^{''}-1,\ldots,\alpha_{1}^{\prime}) \longrightarrow (\alpha_{1}^{\prime}+1,\alpha_{1}^{\prime}+2,\ldots,\alpha_{1}^{''}+1)\]
\[(\alpha_{1}^{''}+1,\alpha_{1}^{''}+2,\ldots,\alpha_{2}^{\prime}-1) \longrightarrow (\alpha_{1}^{''}+2,\alpha_{1}^{''}+3,\ldots,\alpha_{2}^{\prime})\]
\[(\alpha_{2}^{''},\alpha_{2}^{''}-1,\ldots,\alpha_{2}^{\prime}) \longrightarrow (\alpha_{2}^{\prime}+1,\alpha_{2}^{\prime}+2,\ldots,\alpha_{2}^{''}+1)\]
\[\vdots\]
\[(\alpha_{p-1}^{''}+1,\alpha_{p-1}^{''}+2,\ldots,\alpha_{p}^{\prime}-1) \longrightarrow (\alpha_{p-1}^{''}+2,\alpha_{p-1}^{''}+3,\ldots,\alpha_{p}^{\prime})\]
\[(\alpha_{p}^{''}-1,\alpha_{p}^{''}-2,\ldots,\alpha_{p}^{\prime})
\longrightarrow (\alpha_{p}^{\prime}+1,\alpha_{p}^{\prime}+2,\ldots,\alpha_{p}^{''}).\]

In the example in Figure \ref{fig2}, we show one such sequence.

\begin{figure}[h]
\centering
\centering
\includegraphics[scale=0.7]{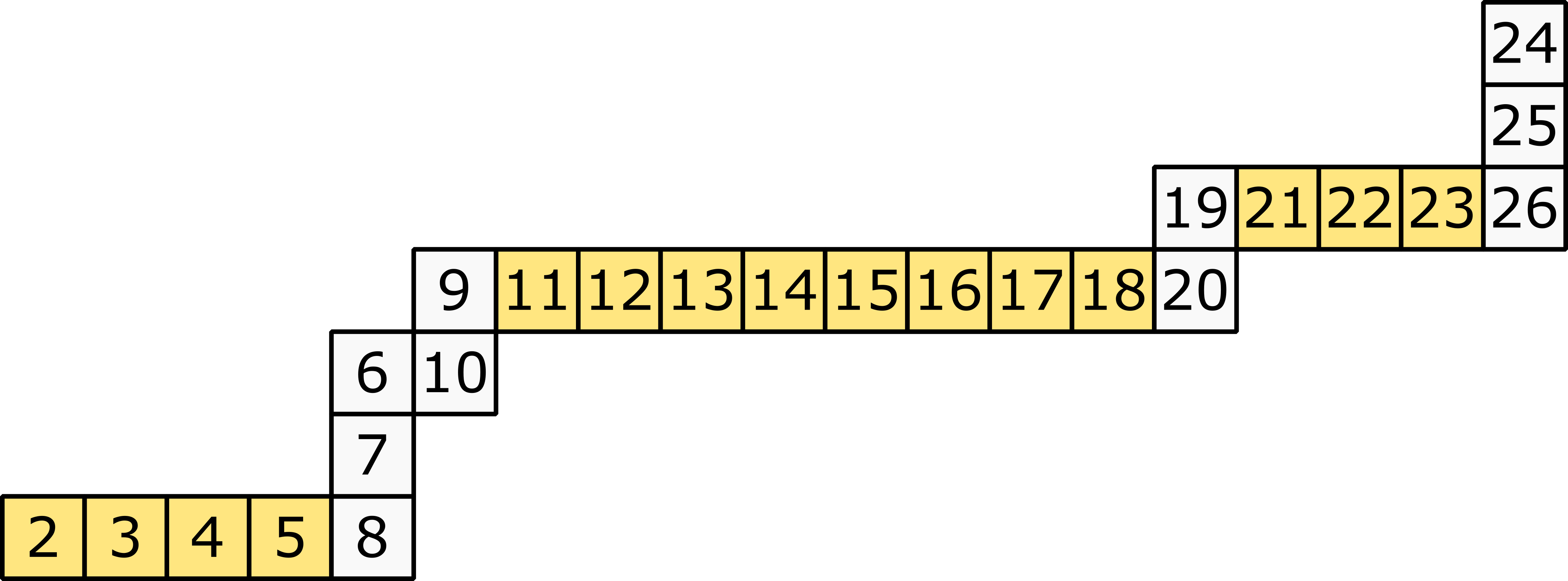}
\caption{Sequence counted by $b_{\alpha_{t}}(I)$.}
\label{fig2}
\end{figure}

Let us replace one-by-one the numbers in the sequence we proposed for $b_{\alpha_{t}}(I)$  with the numbers in the sequence we proposed for $b_{L}(I)$. We replace the numbers in this order: we first replace the smallest number, then the second smallest, the third, and so on. By doing this we preserve the semistandard Young tableau in all median positions. Because we change one number at a time, we get that for every $i: \ L\leq i\leq \alpha_{t}$ there is a sequence counted by $b_{i}(I)$.
\end{proof}

\begin{corollary}
For every integer $k$, the coefficients of $d^{\infty}(I,n)$ in the base $\left(\binom{n+k}{i}\right)_{i=0}^{\infty}$ are integers.
\end{corollary}

\begin{proof}
To prove this we use induction on $k$ in both directions. From Theorem \ref{theo1}, we know that the coefficients of $d^{\infty}(I,n)$ are integers in the base $\left(\binom{n-1}{i}\right)_{i=0}^{\infty}$. Let us suppose that the coefficients of $d^{\infty}(I,n)$ are integers in the base $\left(\binom{n+k}{i}\right)_{i=0}^{\infty}$. We want to prove that they are integers in both the bases $\left(\binom{n+k-1}{i}\right)_{i=0}^{\infty}$ and $\left(\binom{n+k+1}{i}\right)_{i=0}^{\infty}$. From Theorem \ref{theo1}, we know that the degree of $d^{\infty}(I,n)$ is $\alpha_{t}$, so we get:
\[
d^{\infty}(I,n) = a_{\alpha_{t}}\binom{n+k}{\alpha_{t}} +a_{\alpha_{t}-1}\binom{n+k}{\alpha_{t}-1} +\cdots+ a_{0}\binom{n+k}{0}
\]
for some integers $a_{0},a_{1},\ldots,a_{\alpha_{t}} \ (a_{\alpha_{t}}\neq 0)$. Using the identities:
\begin{align*}
\binom{n+k}{i}=\binom{n+k-1}{i}+\binom{n+k-1}{i-1}\\
\binom{n+k}{i}=\sum\limits_{j=0}^{i}(-1)^{i-j}\binom{n+k+1}{j}   
\end{align*}
we can express $d^{\infty}(I,n)$ in both the bases $\left(\binom{n+k-1}{i}\right)_{i=0}^{\infty}$ and $\left(\binom{n+k+1}{i}\right)_{i=0}^{\infty}$ and conclude that the coefficients in both bases are integers.
\end{proof}

Next, we look at the coefficients of $d^{\infty}(I,n)$ in the base $\left(\binom{n}{i}\right)_{i=0}^{\infty}$. We first introduce some notations, which will simplify the reasoning in the following proofs.    

\begin{definition}\label{def6}
For the non-empty set $I=\{\alpha_{1},\ldots,\alpha_{t}\}$, let $x_{i}(I)$ equal the number of sequences $v=(v_{1},v_{2},\ldots,v_{\alpha_{t}})$, which satisfy the following conditions:
\begin{itemize}
\item $\Des(v)=I^{-}$,
\item $v_{j}\in\{2,3,\dots i+1\}$,
\item for every $\ell\in\{2,3,\dots i+1\}$, there exists $j: \ v_{j}=\ell$.
\end{itemize}
\end{definition}

\begin{definition}\label{def7}
For a non-empty set $I=\{\alpha_{1},\ldots,\alpha_{t}\}$, let $y_{i}(I)$ equal the number of sequences $v=(v_{1},v_{2},\ldots,v_{\alpha_{t}})$, which satisfy the following conditions:
\begin{itemize}
\item $\Des(v)=I^{-}$,
\item $v_{j}\in\{1,2,\dots i+1\}$,
\item for every $\ell\in\{1,2,\dots i+1\}$, there exists $j: \ v_{j}=\ell$,
\item $v_{\alpha_{t}}\neq 1$.
\end{itemize}
\end{definition}

\begin{definition}\label{def4}
Let $c_{0}(I),c_{1}(I),c_{2}(I),\ldots$ be the coefficients of $d^{\infty}(I,n)$ in the base  
$\left(\binom{n}{i}\right)_{i=0}^{\infty}$:
\[
d^{\infty}(I,n) = c_{0}\binom{n}{0} + c_{1}\binom{n}{1} + c_{2}\binom{n}{2}+\cdots.
\]
\end{definition}

In the next theorem, we look at the coefficients $c_{i}(I)$. We derive a relation between $c_{i}(I)$ and $b_{i}(I)$ and find the exact values of $c_{i}(I)$ for some $i$.

\begin{theorem}\label{theo2}
Let $t=|I|$ and let $L$ be the length of the longest sequence of consecutive numbers in $I$. Then $c_{i}(I)=(-1)^{i+t}$ for all $i$ such that $0\leq i\leq L$.
\end{theorem}
\begin{proof}
Because we do not have much information about the coefficients \linebreak $c_{0}(I),c_{1}(I),c_{2}(I),\ldots$, we want to express them in terms of $b_{0}(I),b_{1}(I),b_{2}(I),\ldots$. To do this we use the identity $\binom{n}{i}=\binom{n-1}{i}+\binom{n-1}{i-1}$.

\begin{equation*}
\begin{split}
d^{\infty}(I,n) & = b_{\alpha_{t}}(I)\binom{n-1}{\alpha_{t}} +b_{\alpha_{t}-1}(I)\binom{n-1}{\alpha_{t}-1} +\cdots+ b_{L}(I)\binom{n-1}{0}\\
& = b_{\alpha_{t}}(I) \left(\binom{n-1}{\alpha_{t}}+\binom{n-1}{\alpha_{t}-1}\right) + \left(b_{\alpha_{t}-1}(I)-b_{\alpha_{t}}(I)\right) \left(\binom{n-1}{\alpha_{t}-1}+\binom{n-1}{\alpha_{t}-2}\right)+ 
\cdots\\
& \phantom{=} + \sum\limits_{i=1}^{\alpha_{t}}(-1)^{i+1}b_{i}(I)\left(\binom{n-1}{1}+\binom{n-1}{0}\right) +
\sum\limits_{i=0}^{\alpha_{t}}(-1)^{i}b_{i}(I)\binom{n-1}{0}\\
& = b_{\alpha_{t}}(I)\binom{n}{\alpha_{t}} + \cdots
 + \sum\limits_{i=1}^{\alpha_{t}}(-1)^{i+1}b_{i}(I)\binom{n}{1} + \sum\limits_{i=0}^{\alpha_{t}}(-1)^{i}b_{i}(I)\binom{n}{0}.
\end{split}
\end{equation*}

Therefore, we obtain that
\[
c_{k}(I)=\sum\limits_{i=k}^{\alpha_{t}}(-1)^{i+k}b_{i}(I).
\]

Because the coefficients $0=b_{0}(I)=b_{1}(I)=\cdots=b_{L-1}(I)$, we get that
\begin{equation}\label{equ6}
c_{0}(I)=-c_{1}(I)=c_{2}(I)=-c_{3}(I)=\cdots=(-1)^{L}c_{L}(I).    
\end{equation}

Our goal is to compute $c_{0}(I)$.
\[
c_{0}(I)=\sum\limits_{i=0}^{\alpha_{t}}(-1)^{i}b_{i}(I).
\]

From Definitions \ref{def6} and \ref{def7} we get
\begin{equation}\label{equ4}
b_{i}(I)=x_{i}(I)+y_{i}(I).    
\end{equation}

Let us consider sequences with length $\alpha_{t}$ and non-empty descent set $I^{-}$. The number of such sequences whose elements are numbers from the set $\{2,3,\ldots,i+2\}$ is equal to $x_{i+1}(I)$. The value of $x_{i+1}(I)$ is equal to the number of such sequences whose elements are taken from the set $\{1,2,\ldots,i+1\}$. The last number can be divided into two parts: the number of sequences with last element bigger than $1$, which is $y_{i}(I)$, and the number of sequences with last element exactly one $1$, which is $b_{i}(I^{-})$. Therefore, we get the equation
\[
x_{i+1}(I)=y_{i}(I)+b_{i}(I^{-}).
\]

Combining it with Equation \eqref{equ4} we obtain the relation:
\[
b_{i}(I)=x_{i+1}(I)+x_{i}(I)-b_{i}(I^{-}).
\]

Also, it is worth mentioning that
\[
b_{0}(I)=x_{1}(I)-b_{0}(I^{-}),
\]
\[
b_{\alpha_{t}}(I)=x_{\alpha_{t}}(I)-b_{\alpha_{t}}(I^{-}),
\]
because $x_{0}(I)=x_{\alpha_{t}+1}(I)=0$. Using this information, let us compute $c_{0}(I)$.

\begin{align*}
c_{0}(I) & =\sum_{i=0}^{\alpha_{t}} (-1)^{i}b_{i}(I)= \sum_{i=0}^{\alpha_{t}}(-1)^{i}\left(x_{i+1}(I)+x_{i}(I)-b_{i}(I^{-})\right) \\
& = - \sum_{i=0}^{\alpha_{t}}(-1)^{i}b_{i}(I^{-})=-c_{0}(I^{-}).
\end{align*}

We can continue this process until we get to $c_{0}(\{\alpha_{1}\})$. We obtain that
\begin{equation}\label{equ5}
c_{0}(I)=(-1)^{t-1} c_{0}(\{\alpha_{1}\}).    
\end{equation}

So, now our main goal is to calculate $c_{0}(\{\alpha_{1}\})$. We start by computing $d^{\infty}(\{\alpha_{1}\},n)$. As we know from Proposition \ref{prop3}, $d^{\infty}(\{\alpha_{1}\},n)$ is equal to the number of sequences $v=(v_{1},v_{2},\ldots,v_{\alpha_{1}})$, whose elements are numbers from the set $\{1,2,\ldots,n\}$, which have last element $v_{\alpha_{1}}$ bigger than $1$, and which have descent set $\varnothing$. So, the sequence $v$ looks like
\[
1\leq v_{1}\leq v_{2}\leq\cdots\leq v_{\alpha_{1}}\leq n.
\]

We have $\binom{n+\alpha_{t}-1}{\alpha_{1}}-1$ such sequences. We remove one because of the sequence $(1,1,\ldots,1)$, which is the only sequence with last element equal to $1$. Using the following identity, we find the coefficients $c_{i}(\{\alpha_{1}\})$.
\[
-1+\binom{n+\alpha_{t}-1}{\alpha_{1}} = -\binom{n}{0}+\sum\limits_{i=1}^{\alpha_{1}}\binom{\alpha_{1}-1}{\alpha_{1}-i}\binom{n}{i}.
\]

We get that $c_{0}(\{\alpha_{1}\})=-1$. After replacing it in Equation \eqref{equ5} we obtain that $c_{0}(I)=(-1)^{t}$. Finally, if we replace this in Equation \eqref{equ6} we derive:
\[c_{i}(I)=(-1)^{i+t}\text{ \ for \ } 0\leq i\leq L. \]
\end{proof}

The next theorem is a continuation of Theorems \ref{theo1} and \ref{theo2}. We are considering the signs of the coefficients of $d^{\infty}(I,n)$ in the base $\left(\binom{n+k}{i}\right)_{i=0}^{\infty}$ for different values of $k$.

\begin{theorem}
Let $k$ be an integer and the set $I$ be non-empty. We consider the coefficients of $d^{\infty}(I,n)$ in the base $\left(\binom{n+k}{i}\right)_{i=0}^{\infty}$. If $k\leq -1$, all coefficients are non-negative. If $k\geq 0$, there exists at least one negative coefficient.
\end{theorem}

\begin{proof}
From the identity $\binom{n+k}{i}=\binom{n+k-1}{i} + \binom{n+k-1}{i-1}$, it follows that if the coefficients of $d^{\infty}(I,n)$ are non-negative in the base $\left(\binom{n+k}{i}\right)_{i=0}^{\infty}$, they should also be non-negative in the base $\left(\binom{n+k-1}{i}\right)_{i=0}^{\infty}$. From Theorems $\ref{theo1}$ and $\ref{theo2}$, we know that the coefficients in the base $\left(\binom{n-1}{i}\right)_{i=0}^{\infty}$ are non-negative and there exists at least one negative coefficient in the base $\left(\binom{n}{i}\right)_{i=0}^{\infty}$. Therefore, for $k\leq -1$ the coefficients of $d^{\infty}(I,n)$ in the base $\left(\binom{n+k}{i}\right)_{i=0}^{\infty}$ are non-negative integers and for $k\geq 0$ there exists at least one negative coefficient.
\end{proof}

\begin{remark}
For $k\geq 0$, there exist negative coefficients in the representation of $d^{\infty}(I,n)$ in the base $\left(\binom{n+k}{i}\right)_{i=0}^{\infty}$. This means that we cannot give a combinatorial interpretation to these coefficients. Therefore, the base $\left(\binom{n-1}{i}\right)_{i=0}^{\infty}$ is optimal for assigning a combinatorial interpretation to the coefficients.
\end{remark}

\section{Acknowledgments} 

The majority of this paper was done during RSI (Research Science Institute) in the summer of 2021. First I want to thank my mentor at RSI, Pakawut Jiradilok. He proposed this topic and guided me through it. Also, he gave useful advice for the mathematical part and helped a lot with the editing of the paper. Many thanks to Dr. John Rickert, Dr. Tanya Khovanova, Prof. David Jerison, \linebreak Prof. Ankur Moitra, Yunseo Choi, Dr. Jenny Sendova, Angelin Mattew, \linebreak Viktor Kolev, Miles Edwards,  Dimitar Dimitrov, and Martin Dimitrov who played a huge role in helping me achieve my goals during RSI. Last but not least, I am very grateful to HSSIMI and Sts. Cyril and Methodius International Foundation for making my participation in RSI possible.

\bibliographystyle{alpha}
\bibliography{biblio}

\end{document}